\documentclass{article}

\usepackage{amssymb,amsmath,mathtools,xcolor,graphicx,xspace,colortbl,ragged2e,rotating} % 
\graphicspath{{Harizanov_Srinivasan__Cohesive_Powers__March_29_2023_graphics/}{Harizanov_Srinivasan__Cohesive_Powers__March_29_2023_tcache/}{Harizanov_Srinivasan__Cohesive_Powers__March_29_2023_gcache/}}
\DeclareGraphicsExtensions{.pdf,.eps,.ps,.png,.jpg,.jpeg}\newtheorem {theorem}{Theorem}

\newtheorem {corollary}[theorem]{Corollary}

\newtheorem {definition}[theorem]{Definition}

\newtheorem {lemma}[theorem]{Lemma}

\newtheorem {proposition}[theorem]{Proposition}

\newenvironment {proof}[1][Proof]{\noindent \textbf {#1.} }{\ \rule {0.5em}{0.5em}}
\begin{document}

\title{Cohesive Powers of Structures\protect\footnote{Acknowledgements\par
The authors gratefully acknowledge support of FRG NSF grant DMS-2152095.
}\textsuperscript {}}

\author{Valentina Harizanov \\Department
of Mathematics\\The George Washington University\\Washington,
DC 20052, USA\\harizanv@gwu.edu
\and Keshav Srinivasan\\Department
of Mathematics\\The George Washington University\\Washington,
DC 20052, USA\\ksrinivasan@gwu.edu
\and 
}
\maketitle

\begin{abstract}A cohesive power of a structure is an effective analog of the classical ultrapower of a structure. We start with a computable structure, and consider its countable ultrapower over a cohesive set of natural numbers. A cohesive set is an infinite set of natural numbers that is indecomposable with respect to computably enumerable sets. It plays the role of an ultrafilter, and the elements of a cohesive power are the equivalence classes of certain partial computable functions. Thus, unlike many classical ultrapowers, a cohesive power is a countable structure. In this paper we focus on cohesive powers of graphs, equivalence structures, and computable structures with a single unary function satisfying various properties, which can also be viewed as directed graphs. For these computable structures, we investigate the isomorphism types of their cohesive powers, as well as the properties of cohesive powers when they are not isomorphic to the original structure.

\bigskip


\bigskip

\end{abstract}

\section{Introduction and Preliminaries
}

We consider a computability-theoretic product construction for structures. We start with a uniformly computable sequence of structures for the same computable language, and in their Cartesian product consider partial computable sequences modulo a fixed cohesive set of natural  numbers. A cohesive set is an infinite set of natural numbers, which is indecomposable with respect to computably enumerable sets. In this paper we focus on effective products that are powers of a single computable structure. Some cohesive sets are the complements of maximal sets. Co-maximal powers arose naturally in the study of the automorphisms of the lattice of computably enumerable vector spaces. In particular, Dimitrov \cite{D4} used cohesive powers of fields to characterize
principal filters of quasimaximal vector spaces.  He later introduced in \cite{D3} the notion of a cohesive power of a computable structure in general. A cohesive power construction produces a countable structure because the elements are represented by partial computable functions. In some cases partial computable functions can be replaced by computable functions. 

The motivation for cohesive powers dates back to Skolem's construction of a countable non-standard model of arithmetic (see  \cite{DH2}) where, instead of building a structure from all functions on natural numbers, he used only arithmetical functions.  Skolem's idea was further developed in the study of models of fragments of arithmetic by Feferman, Scott and Tennenbaum \cite{FST}, Lerman \cite{L} Hirschfeld and Wheeler \cite{Hi,HW} and McLaughlin \cite{ML1,ML2,ML3}. In \cite{Ne}, Nelson investigated recursive saturation of effective ultraproducts.

Cohesive power construction allows us to obtain countable models with interesting properties. A cohesive power of a structure $\mathcal{A}$ may not be elementarily equivalent to $\mathcal{A}$. Dimitrov established a restricted version of \L o\'{s}'s theorem for cohesive powers. However, additional decidability on the structure plays a significant role in increasing satisfiability of the same sentences in the ultrapower.

Recall the following notions from computability theory. The complement of a set $X \subseteq \omega $ is denoted by $\overline{X}$. We write $ \subseteq ^{ \ast }$ for inclusion up to finitely many elements. By c.e.\ we abbreviate computably enumerable. An infinite set $C \subseteq \omega $ is \emph{cohesive} if for every c.e.\ set $W\text{,}$ either $W \cap C$ or $\overline{W} \cap C$ is finite. If $W \cap C$ is infinite, then $C \subseteq ^{ \ast }W ,$ and if  $\overline{W} \cap C$ is infinite, then $C \subseteq ^{ \ast }\overline{W}$. Clearly, an infinite subset of a cohesive set is cohesive. It follows that if a cohesive set $C$ is contained in the union of finitely many c.e.\ sets, up to finitely many elements, then it is contained in one of them, up to finitely many elements. That is, because $C$ must have an infinite intersection with at least one of the finitely many c.e. sets in the union. It can be shown that every infinite set of natural numbers has a cohesive subset. Hence there are continuum many cohesive sets. Some cohesive sets are complements of c.e. sets. A set $E \subseteq \omega $ is \emph{maximal} iff $E$ is c.e.\ and $\overline{E}$ is cohesive.

  If $L$ is the language of a structure $\mathcal{A}$ with domain $A$, then $L_{A}$ is the language $L$ expanded by adding a constant symbol for every $a \in A$, and $\mathcal{A}_{A} =(\mathcal{A} ,a)_{a \in A}$ is the corresponding expansion of $\mathcal{A}$ to $L_{A}$. The \emph{atomic}\emph{ diagram}
of $\mathcal{A}$ is the set of all atomic and negations of  atomic sentences of $L_{A}$ true in $\mathcal{A}_{A}$.  A countable structure for\  a computable language $L$ is \emph{computable} if its domain is computable and its atomic diagram is computable or, equivalently, its functions and relations are uniformly computable. The \emph{elementar}\emph{y diagram} (or complete diagram) of $\mathcal{A}$, denoted by $D^{c} (\mathcal{A})$, is the set of all first-order sentences of $L_{A}$ that are true in $\mathcal{A}_{A}$. A $\Sigma _{n}^{0}$\emph{ diagram} of $\mathcal{A}$ is the set of all $\Sigma _{n}^{0}$ sentences in $D^{c} (\mathcal{A})$. A structure is \emph{decidable} if its doman is computable and its elementary diagram is computable. A structure is $n$-\emph{decidable} if its doman is computable and its $\Sigma _{n}^{0}$ diagram is computable. In particular, computable structures are the same as $0$-decidable structures. 

We will now give a definition of a cohesive product of computable structures which appears in \cite{DH0}. By $ \simeq $ to denote the equality of partial functions.

\begin{definition}
Let $L$ be a computable language. Let $(\mathcal{A}_{i})_{i \in \omega }$ be a uniformly computable sequence of computable structures in $L\text{,}$ with uniformly computable sequence of domains $(A_{i})$\textsubscript {$i \in \omega $}. Let $C \subseteq \omega $ be a cohesive set. The \emph{cohesive product }$\mathcal{B}$\emph{\ of }$\mathcal{A}_{i}$\emph{\ over }$C$, in symbols $\mathcal{B} =\prod _{C}\mathcal{A}_{i}$, is a structure defined as follows.

\begin{enumerate}
\item Let 

\begin{equation*}D =\{\psi  \mid \psi  :\omega  \rightarrow \bigcup _{i \in \omega }A_{i}~\mathrm{i} \mathrm{s}~\mathrm{a}~\mathrm{p} \mathrm{a} \mathrm{r} \mathrm{t} \mathrm{i} \mathrm{a} \mathrm{l}~\mathrm{c} \mathrm{o} \mathrm{m} \mathrm{p} \mathrm{u} \mathrm{t} \mathrm{a} \mathrm{b} \mathrm{l} \mathrm{e}~\mathrm{f} \mathrm{u} \mathrm{n} \mathrm{c} \mathrm{t} \mathrm{i} \mathrm{o} \mathrm{n}\text{ } \wedge \ C \subseteq ^{ \ast }dom(\psi )\} .
\end{equation*} For $\psi _{1} ,\psi _{2} \in D$, let
\begin{equation*}\psi _{1}  =_{C}\psi _{2}\text{\quad }\text{\textrm{iff}}\text{\quad }C  \subseteq ^{ \ast }\{i :\psi _{1} (i)\downarrow  =\psi _{2} (i)\downarrow \}\text{.}
\end{equation*}The domain of $\prod _{C}\mathcal{A}_{i}$ is the quotient set $D/_{ =_{C}}$ and is denoted here by $B\text{.}$ \smallskip 

\item If $f \in L$ is an $n$-ary function symbol, then $f^{\mathcal{B}}$ is an $n$-ary function on $B$ such that for every $[\psi _{1}] ,\ldots  ,[\psi _{n}] \in B$, we have
\begin{equation*}f^{\mathcal{B}}([\psi _{1}] ,\ldots  ,[\psi _{n}]) =[\psi ]\text{\quad }\text{\textrm{iff}}\text{\quad }\left ( \forall i \in \omega \right )\left [\psi  (i) \simeq f^{\mathcal{A}_{i}} (\psi _{1} (i) ,\ldots  ,\psi _{n} (i))\right ]\text{.}
\end{equation*} 

\item If $R \in L$ is an $m$-ary relation symbol, then $R^{\mathcal{B}}$ is an $m$-ary relation on $B$ such that for every $[\psi _{1}] ,\ldots  ,[\psi _{m}] \in B$,
\begin{equation*}R^{\mathcal{B}}([\psi _{1}] ,\ldots  ,[\psi _{m}])\text{\quad }\text{\textrm{iff}}\text{\quad }C  \subseteq ^{ \ast }\{i \in \omega  \mid R^{\mathcal{A}_{i}} (\psi _{1} (i) ,\ldots  ,\psi _{m} (i))\}\text{.}
\end{equation*} 

\item If
$c \in L$ is a constant symbol, then $c^{\mathcal{B}}$ is the equivalence class (with respect to $ =_{C}$) of the computable function $g :\omega  \rightarrow \bigcup _{i \in \omega }A_{i}$ such that $g (i) =c^{\mathcal{A}_{i}}\text{,}$ for each $i \in \omega $. \medskip \end{enumerate}
\end{definition}

If $C$ is co-c.e., then for every $[\psi ] \in \prod _{C}\mathcal{A}_{i}$ there is a computable function $f$ such that $\left [f\right ] =\left [\psi \right ]$. That is, for $(a_{i})_{i} \in \Pi _{i}A_{i}$ which is a fixed computable sequence, define \bigskip  

$f (i) =\left \{\begin{array}{cc}\psi  (i) & \text{if }\psi  (i)\downarrow \text{first,} \\
a_{i} & \text{if }i\text{  is enumerated into }\overline{C}\text{  first.}\end{array}\right .$

\bigskip

If $\mathcal{A}_{i} =\mathcal{A}$ for $i \in \omega $, then $\prod _{C}\mathcal{A}_{i}$ is called the \emph{cohesive power of }$\mathcal{A}$\emph{ over }$C$ and is denoted by $\prod _{C}\mathcal{A} .$ 

An \emph{embedding} of a structure $\mathcal{A}$ into a structure $\mathcal{B}$ is an isomorphism between $\mathcal{A}$ and a substructure of $\mathcal{B}$. A structure $\mathcal{A}$ embeds into its cohesive power $\mathcal{B} =\prod _{C}\mathcal{A}$. For $a \in A$ let $[c_{a}] \in B$ be the equivalence class of the total function $c_{a}$ such that $c_{a} (i) =a$ for every $i \in \omega $. The function $d :A \rightarrow B$ such that $d (a) =[c_{a}]$ is called the \emph{canonical embedding} of $\mathcal{A}$ into $\mathcal{B}$.

In \cite{DH0,LinOrd} we provide variants of \L o\'{s}'s theorem for cohesive products of uniformly computable and more generally uniformly $n$-decidable structures. For example, every $\Sigma _{n +3}^{0}$ sentence true in an $n$-decidable structure is also true in its cohesive powers. In particular, we have the following theorem for cohesive powers of computable structures.    \medskip

\begin{theorem}
(Dimitrov \cite{D3}) \label{rumen}Let $\mathcal{B} =\prod _{C}\mathcal{A}$ be a cohesive power of a computable structure $\mathcal{A} .$ Let $C$ be a cohesive set. 
\begin{enumerate}
\item If $\sigma $ is a $\Pi _{2}^{0}$ (or $\Sigma _{2}^{0})$ sentence in $L$, then $\mathcal{B} \models \sigma $ iff $\mathcal{A} \models \sigma  .$ 

\item
If $\sigma $ is a $\Pi _{3}^{0}$ sentence in $L$, then $\mathcal{B} \models \sigma $ implies  $\mathcal{A} \models \sigma  .$

By contrapositive, if $\sigma $ is a $\Sigma _{3}^{0}$ sentence in $L$, then $\mathcal{A} \models \sigma $ implies $\mathcal{B} \models \sigma $

\end{enumerate}
\end{theorem}

The converse of part (2) in the previous theorem does not hold. The first such counterexample was produced by Feferman, Scott and Tennenbaum in their result in \cite{FST} that no cohesive power of the standard model of arithmetic is a model of Peano arithmetic. There is a $\Pi _{3}^{0}$ sentence involving Kleene's $T$ predicate that is true in the standard model of arithmetic $\mathcal{N}$ but is false in every cohesive power of $\mathcal{N}$ (see \cite{L}). More recently, in \cite{LinOrd}, we produced natural examples of such sentences concerning linear orders.   

In \cite{D3}, Dimitrov established that if $\mathcal{A}$ is a decidable structure, then $\mathcal{A}$ and $\prod _{C}\mathcal{A}$ satisfy the same first-order sentences, i.e., they are elementarily equivalent. A corresponding result has been formulated for $n$-decidable structures in \cite{LinOrd1}.

An equivalence structure $\mathcal{A} =(A ,E^{\mathcal{A}})$ consists of a set $A$ with a binary relation $E^{\mathcal{A}}$ that is reflexive, symmetric, and transitive. An equivalence structure $\mathcal{A}$ is \emph{computable} if $A$ is a computable set and $E^{\mathcal{A}}$ is a computable relation. An application of Theorem \ref{rumen} is that if $\mathcal{A}$ is a computable equivalence structure, then so is $\prod _{C}\mathcal{A}$. That is because the theory of equivalence structures is $\Pi _{1}^{0}$-axiomatizable. Similarly, a cohesive power of a computable field is a field. In \cite{DHMM}, we investigated cohesive powers of the field $\mathbb{Q}$ of rational numbers over co-maximal sets. For example, we proved that two cohesive powers of $\mathbb{Q}$ over co-maximal sets are isomorphic iff the maximal sets have the same $m$-degree.

Dimitrov \cite{D3} showed that if $\mathcal{A}$ is a finite structure, then $\prod _{C}\mathcal{A} \cong $ $\mathcal{A}$. Also, if $\mathcal{A}$ and $\mathcal{B}$ are computably isomorphic structures and $C$ is a cohesive set, then $\Pi _{C}\mathcal{A} \cong \Pi _{C}\mathcal{B}$ (for a proof see \cite{LinOrd}). A computable structure $\mathcal{A}$ is called \emph{computably categorical} if every computable isomorphic structure is computably isomorphic to $\mathcal{A}$.

A computable (infinitary) language is more expressive than the usual finitary first-order language. For a computable ordinal $\alpha $, Ash defined computable $\Sigma _{\alpha }$ and $\Pi _{\alpha }$ formulas of $L_{\omega _{1} \omega }$ recursively and simultaneously and together with their G{\"o}del numbers. For the natural numbers we roughly have the following classification of formulas. Computable $\Sigma _{0}$ and $\Pi _{0}$ formulas are just the finitary quantifier-free formulas. For $n >1$, a \emph{computable }$\Pi _{n}$\emph{ formula} is a c.e.\ conjunction of formulas $ \forall \overline{u}\, \phi  (\overline{x} ,\overline{u})$, where $\phi $ is a computable $\Sigma _{m}$ formula for some $m <n$. Dually, a \ \emph{computable }$\Sigma _{n}$ \emph{formula} is a c.e.\ disjunction of formulas $ \exists \overline{v}\, \theta  (\overline{y} ,\overline{v})\text{}$, where $\theta $ is a computable $\Pi _{m}$ formula for some $m <n$. (See \cite{FHM}.) For more on computability theory see \cite{RS}. By $ \langle k ,n \rangle $ we denote a computable bijection from $\omega ^{2}$ onto $\omega  ,$ which is strictly increasing with respect to each coordinate and such that $k ,n_{} \leq  \langle k ,n \rangle $. 

\smallskip \medskip

This paper is a greatly expanded version of the preliminary work in \cite{HS} to appear in the proceedings  following the Fall Western Sectional Meeting of the AMS during October 23--24, 2021. There, we studied cohesive powers of certain graphs and equivalence structures. For example, we showed that every computable graph can be embedded into a cohesive power of a strongly locally finite graph. Here, we also investigate cohesive powers of computable structures with a single unary function satisfying various properties, called injection structures, two-to-one structures, and (2,0)-to-one structures. We further study cohesive powers of partial injection structures viewed as relational structures. We characterize the isomorphism types of cohesive powers of these computable structures, and use computable (infinitary) language to describe the properties of cohesive powers when they are not isomorphic to the original structure.

\section{Cohesive powers of graphs
}
A \emph{graph} (or undirected graph) $(V ,E)$ is a nonempty set $V$ of vertices with a symmetric binary relation $E$ (also called the edge relation), so it can be axiomatized by the following universal sentence:\medskip

$ \forall x \forall y[E(x ,y) \Rightarrow E(y ,x)]$. \medskip

\noindent Hence, by Theorem \ref{rumen} a cohesive power of a graph is a graph.

 If $(x ,y) \in E ,$ then vertices $x$ and $y$ are adjacent to each other. The \emph{degree of a vertex} is the number of vertices it is adjacent to. A graph $G$ is called \emph{locally finite} if the degree of each vertex in $G$ is finite. A graph $G$ is \emph{strongly locally finite }if all connected components of $G$ are finite. In \cite{CKL}, a criterion was obtained for computable categoricity of certain strongly locally finite computable graphs. 

The \emph{disjoint union} of graphs $(V_{1} ,E_{1})$ and $(V_{2} ,E_{2})$ where $V_{1} \cap V_{2} = \varnothing $ is a graph $(V_{1} \cup V_{2} ,E_{1} \cup E_{2})$. Hence there are no edges between $V_{1}$ and $V_{2}$. We write $(V_{1} ,E_{1}) \amalg (V_{2} ,E_{2})$ and also view it as  decomposition into a disjoint union. If $V_{1} \cap V_{2} = \varnothing  ,$ then a union of graphs $G_{1} =(V_{1} ,E_{1})$ and $G_{2} =(V_{2} ,E_{2})$ is any graph $G =(V_{1} \cup V_{2} ,E)$ such that $E_{1} \cup E_{2} \subseteq E$ and $E \upharpoonright (V_{i} \times V_{i}) =E_{i}$ for $i =1 ,2$. We simply write $G =G_{1} \cup G_{2} .$

The following result demonstrates the universal feature with respect to embeddability into cohesive powers of certain computable graphs.

\begin{theorem}
Let $G$ be a computable graph. Let $C$ be a cohesive set. Then there is a computable, strongly locally finite graph $\mathcal{A}$ such that $\Pi _{C}\mathcal{A}$ is isomorphic to the union $G \cup H$ for some graph $H$ or $H = \varnothing $.
\end{theorem}

\begin{proof}
If $G$ is finite, then $G$ is strongly locally finite and $G \cong \Pi _{C}G$. 

Let $\mathbb{N}^{ +} =\omega  -\{0\} .$ Now, assume that a graph $G =(V ,R_{})$ is infinite and fix a computable enumeration $f$ of its vertices  $V =\{f(n) :$ $n \in \mathbb{N}^{ +}\} .$ Using this enumeration we will build a computable graph $\mathcal{A}$ with domain $\mathbb{N}^{ +}$\textsuperscript and edge set $E$. Consider vertices $a ,b \in \mathbb{N}^{ +} .$ If $a ,b$ can be written as $a =\frac{k(k +1)}{2} +m$ and $b =\frac{k(k +1)}{2} +n$ for some $k ,m ,n$ such that $1 \leq m ,n \leq k +1 ,$ then let \medskip

 $(a ,b) \in E \Leftrightarrow (f(m) ,f(n)) \in R .$ \medskip  

\noindent If there are no $k ,m ,n$ as above, then let $(a ,b) \notin E$. Note that $\frac{k(k +1)}{2} =1 +2 + \cdots  +k$, so the idea is to divide the natural numbers into segments of length $1 ,2 ,\ldots $. Clearly, $\mathcal{A}$ is a computable, strongly locally finite graph.

For $n \in N^{ +}$ we define functions $\psi _{n}$ by $\psi _{n}(x) =\frac{x(x +1)}{2} +n$. Hence each $[\psi _{n}]$ is an element of the cohesive power $\Pi _{C}\mathcal{A} .$ Consider the subgraph $\mathcal{S}$ of $\mathcal{A}$ with the vertex set $\{[\psi _{n}] :n \in \mathbb{N}^{ +}\}$. Consider a function $\rho  :\{[\psi _{n}] :n \in \mathbb{N}^{ +}\} \rightarrow V$ defined by $\rho ([\psi _{n}]) =f(n)$. Then we can show that $\rho $ is a graph isomorphism (see \cite{HS}), so $\mathcal{S}$ is isomorphic to $G .$ Hence $\Pi _{C}\mathcal{A}$ is isomorphic to the union $G \cup H$ for some graph $H$.  
\end{proof}

\bigskip If a computable graph $G$ is locally finite, we have a stronger result.

\begin{theorem}
Let $G$ be an infinite computable graph that is locally finite. Let $C$ be a cohesive set. Then there is a computable, strongly locally finite graph $\mathcal{A}$ such that $\Pi _{C}\mathcal{A}$ is isomorphic to the disjoint union $G \sqcup H$ for some graph $H$.
\end{theorem}

\begin{proof}
Let $G =(V ,R_{})$.  Let $f ,\mathcal{A} ,\psi _{n} ,\rho $ be defined as in the proof of the previous theorem. Let $\varphi $ be a partial computable function such that $[\varphi ] \in \Pi _{C}A$ and  $E([\varphi ] ,[\psi _{m}])$ for some $m \geq 1.$ Then $C \subseteq ^{ \ast }\{i \in \omega  :\varphi (i)\downarrow  \wedge \ (\varphi (i) ,\psi _{m}(i)) \in E\}$.  Since $G$ is locally finite, we have that $\{i \in \omega  :\varphi (i)\downarrow  \wedge \ (\varphi (i) ,\psi _{m}(i)) \in E\}$ is the following finite disjoint union of c.e.\ sets: \medskip

$\coprod _{n :f(m)Rf(n)}\{i \in \omega  :\varphi (i) =\frac{i(i +1)}{2} +n$ where $1 \leq m ,n \leq i +1\}$. \medskip

\noindent Since $C$ is cohesive, there is some $n_{0}$ such that \medskip

$C \subseteq ^{ \ast }\{i :\varphi (i)\downarrow  =\frac{i(i +1)}{2} +n_{0}$$\} \subseteq \{i :\varphi (i)\downarrow  =\psi _{n_{0}}(i)\}$. \medskip

\noindent Hence $[\varphi ] =[\psi _{n_{0}}] .$ Thus, $\Pi _{C}\mathcal{A}$ is isomorphic to the disjoint union $G \sqcup (\Pi _{C}\mathcal{A} -\{[\psi _{m}] :m \geq 1\})$. \medskip 
\end{proof}

\section{Cohesive powers of equivalence relations
}

Let $\mathcal{A} =(A ,E^{\mathcal{A}})$ be an equivalence structure $\mathcal{A}$. The equivalence class of $a \in A$ is
\begin{equation*}eqv^{\mathcal{A}}(a) =\{x \in A :x E^{\mathcal{A}} a\}\text{.}
\end{equation*}We generally omit the superscript when it can be inferred from the context.

\begin{definition}\begin{enumerate}
\item [$($i$)$] Let $\mathcal{A}$ be an equivalence relation. The \emph{character of }$\mathcal{A}$ is the set
\begin{equation*}\chi  (\mathcal{A}) =\{ \langle k ,n \rangle  :\;n ,k >0\;\text{and}\;\mathcal{A}\;\text{has}\;\text{at least $n$}\ \text{equivalence classes of size}\;k\}\text{.}
\end{equation*}

\item [$($ii$)$] We say that $\mathcal{A}$ has \emph{bounded}\emph{ character} if there is some finite $k$ such that all finite equivalence classes of $\mathcal{A}$ have size at most $k$. \end{enumerate}
\end{definition}

\noindent Clearly, two countable equivalence structures are isomorphic if they have the same character and the same number of infinite equivalence classes.

For a set $X$, by $card(X)$  or $\left \vert X\right \vert $ we denote the size of $X .$ Let
\begin{equation*}I n f^{\mathcal{A}} =\{a :eqv^{\mathcal{A}}(a)\text{  is infinite}\}\text{  and }F i n^{\mathcal{A}} =\{a :eqv^{\mathcal{A}}(a)\text{  is finite}\}\text{.}
\end{equation*}The following lemma from \cite{CCHM} gives us some important complexities.

\begin{lemma}
 For any computable equivalence structure $\mathcal{A}$: 

\begin{enumerate}
\item [$($a]$)\{ \langle k$$ ,a \rangle  :c a r d (eqv^{\mathcal{A}}(a))^{\mathcal{}} \leq k\}$ is a $\Pi _{1}^{0}$ set, and $\{ \langle k ,a \rangle  :c a r d (eqv^{\mathcal{A}}(a))^{\mathcal{}} \geq k$ is a $\Sigma _{1}^{0}$ set; 

\item [$($b$)$]$I n f^{\mathcal{A}}$ is a $\Pi _{2}^{0}$ set, and $F i n^{\mathcal{A}}$ is a $\Sigma _{2}^{0}$ set; 

\item [$($c$)$]$\chi  (\mathcal{A})$ is a $\Sigma _{2}^{0}$ set. \end{enumerate}
\end{lemma}

We say that a subset $K$ of $\omega $ is a \emph{character} if there is some equivalence structure with character $K$. This is the same as saying that $K \subseteq  \langle \omega  -\{0\} \rangle  \times  \langle \omega  -0\} \rangle $, for all $n >0$ and $k$,
\begin{equation*} \langle k ,n +1 \rangle  \in K \Rightarrow  \langle k ,n \rangle  \in K\text{.}
\end{equation*}

It was shown in \cite{CCHM} that for any $\Sigma _{2}^{0}$ character $K$, there is a computable equivalence structure $\mathcal{A}$ with character $K$, which has infinitely many infinite equivalence classes while $F i n^{\mathcal{A}}$ is a $\Pi _{1}^{0}$ set.

\begin{theorem}
\cite{CCHM} Let $\mathcal{A}$ be a computable equivalence structure. The structure $\mathcal{A}$ is computably categorical iff it is one of the following types: 

\begin{enumerate}
\item $\mathcal{A}$ has only finitely many finite equivalence classes; 

\item $\mathcal{A}$ has finitely many infinite classes, bounded character, and at most one finite $k$ such that there are infinitely many classes of size $k$. \end{enumerate}
\end{theorem}

Hence if $\mathcal{A}$ is an equivalence structure as in the previous theorem and $\mathcal{D}$ is a computable structure isomorphic to $\mathcal{A} ,$ then for any cohesive set $C$ we have $\Pi _{C}\mathcal{A} \cong \Pi _{C}\mathcal{D} .$

\bigskip

\begin{proposition}
 \label{basiceqv}Let $\ \mathcal{A}$ be a computable equivalence structure. Let $C$ be a cohesive set and let $\mathcal{B} =\Pi _{C}\mathcal{A} .$ \medskip

(a) Then $\chi (\mathcal{B}) =\chi (\mathcal{A}) .$\medskip  

(b) If $\mathcal{A}$ has infinitely many infinite equivalence classes, then $\mathcal{B}$ has infinitely many infinite equivalence classes.\medskip  

If $\mathcal{A}$ has exactly $n$ infinite equivalence classes, then $\mathcal{B}$ has at least $n$ infinite equivalence classes.

\medskip 

(c) If $\mathcal{A}$ has infinitely many infinite equivalence classes, then $\Pi _{C}\mathcal{A} \cong \mathcal{A} .$\medskip

(d) If $\ \mathcal{A}$ has finitely many finite equivalence classes and no infinite equivalence classes, then  $\Pi _{C}\mathcal{A} \cong \mathcal{A} .$
\end{proposition}

\begin{proof}
 (a) The character of an equivalence structure is definable by a $\Sigma _{2}^{0}$ sentence (see \cite{CCHM}). \medskip

(b) This holds since there is an embedding of $\mathcal{A}$ into $\mathcal{B}$.\medskip

(c) This follows from (a) and (b). That is, if $\mathcal{A}$ has infinitely many infinite equivalence classes, then $\Pi _{C}\mathcal{A}$ also has infinitely many infinite equivalence classes, and since $\mathcal{A}$ and $\Pi _{C}\mathcal{A}$ have the same character, they are isomorphic. \medskip

(d) This is true since $\mathcal{A}$ is a finite structure.
\end{proof}

\begin{theorem}
Let $\ \mathcal{A} =(A ,E)$ be a computable equivalence structure. Let $C$ be a cohesive set. If $\mathcal{A}$ has a bounded character, then $\Pi _{C}\mathcal{A} \cong \mathcal{A} .$  
\end{theorem}

\begin{proof}
  Since the character of $\mathcal{A}$ is bounded, if it is nonempty, let $k \in \omega $ be the largest size of a finite equivalence class of $\mathcal{A}$. If the character of $\mathcal{A}$ is empty, let $k =0$. Recall that $\mathcal{A}$ and $\Pi _{C}\mathcal{A}$ have the same character. If $\mathcal{A}$ has infinitely many infinite equivalence classes, then $\mathcal{A}$ and $\Pi _{C}\mathcal{A}$ are isomorphic. Thus, assume that $\mathcal{A}$ has at most finitely many infinite equivalence classes.    

If $\mathcal{A}$ has no infinite equivalence class, then $\mathcal{A}$ satisfies the following $\Pi _{1}^{0}$ sentence, saying that there are no $k +1$ non-equivalent elements:\smallskip

\begin{center}
$( \forall x_{1}) \cdots ( \forall x_{k +1})[\bigvee _{1 \leq i <j \leq k +1}x_{i}Ex_{j}]$\smallskip
\end{center}\par
\noindent Since $\Pi _{C}\mathcal{A}$ satisfies the same sentence, it has no infinite equivalence classes, so $\mathcal{A}$ and $\Pi _{C}\mathcal{A}$ are isomorphic. 

Thus, assume that $\mathcal{A}$ has $m$ infinite equivalence classes, where $m \in \omega $ and $m >0$. Hence $\mathcal{A}$ satisfies the following $\Sigma _{2}^{0}$ sentence, saying that there are exactly $m$ equivalence classes with at least $k +1$ elements, hence infinite. We will use notation $x_{i}^{l}$ for variables.

\begin{center}
$( \exists x_{1}^{1}) \cdots ( \exists x_{k +1}^{1}) \cdots ( \exists x_{1}^{m}) \cdots ( \exists x_{k +1}^{m})( \forall y_{1}) \cdots ( \forall y_{k +1})[{\textstyle\bigwedge _{\begin{array}{c}1 \leq l \leq m \\
1 <i <j \leq k +1\end{array}}}(x_{i}^{l} \neq x_{j}^{l} \wedge x_{i}^{l}Ex_{j}^{l}) \wedge (\bigwedge _{1 \leq i <j \leq k +1}$ $(y_{i} \neq y_{j} \wedge y_{i}Ey_{j}) \Rightarrow \bigvee _{l =1 ,\ldots  ,m}y_{1}Ex_{1}^{l})]$

\end{center}\par
\noindent Hence $\Pi _{C}\mathcal{A}$ satisfies the same sentence, so it has exactly $m$ infinite equivalence classes, so it is isomorphic to $\mathcal{A} .$
\end{proof}

\begin{theorem}
Let $\ \mathcal{A} =(A ,E)$ be a computable equivalence structure. Let $C$ be a cohesive set. If $\mathcal{A}$ has an unbounded character, then $\Pi _{C}\mathcal{A}$ has infinitely many infinite equivalence classes. 

Hence if $\mathcal{A}$  has an unbounded character and finitely many (possibly zero) infinite equivalence classes, then $\Pi _{C}\mathcal{A} \ncong \mathcal{A} .$
\end{theorem}

\begin{proof}
  Without loss of generality, we will assume that $A =\omega  .$ Hence we can order the elements in $A$ by the usual ordering of the natural numbers. Let $\mathcal{B} =\Pi _{C}\mathcal{A} ,$ and $\mathcal{B} =(B ,E_{B})$. Choose a computable sequence of elements in $A :\medskip $

$a(0) ,a(1) ,a(2) ,a(3) ,\ldots $ \medskip

\noindent such that they all belong to distinct equivalence classes, and for every $k ,$ the equivalence class of $a(k)$ has $ >k$ elements. Since $E$ is computable and the character of $\mathcal{A}$ is unbounded, such a sequence can be obtained by enumerating $A$ and checking the conditions. We can think of $a( \langle m ,i \rangle )$ as a representative of the equivalence class in $\mathcal{A}$ where a partial function $\psi _{m , \ast }$ might take its $i^{th}$ coordinate value (for $i \geq 0$). That is, define partial computable functions $\psi _{m ,n}(i)$, for $m ,i \geq 0$ and $n \geq 1 ,$ as follows:\bigskip

$\psi _{m ,n} (i) =\left \{\begin{array}{cc}c_{n} & c_{n}\text{  is the }n^{th}\text{  element in }eqv(a \langle m ,i \rangle )\text{  if it exists;
} \\
\uparrow  & \text{otherwise.}\end{array}\right .$

\bigskip \noindent Every $\psi _{m ,n}$ is defined for all except possibly finitely many initial values, so $C \subseteq ^{ \ast }dom(\psi _{m ,n})$. Hence $[\psi _{m ,n}] \in B .$ \medskip

Fix $m .$ Let $n_{1} \neq n_{2}$. Then we have that $[\psi _{m ,n_{1}}] \neq [\psi _{m ,n_{2}}]$ since $\left \vert eqv(a( \langle m ,i \rangle )\right \vert  > \langle m ,i \rangle  \geq i$, so starting with some $i ,\text{ }eqv(a \langle m ,i \rangle )\text{  will have the }n_{1}^{th}$ and the $n_{2}^{th}$ elements and they will be distinct. Also, $[\psi _{m ,n_{1}}]E_{B}[\psi _{m ,n_{2}}]$  since the values of $\psi _{m ,n_{1}}$ and $\psi _{m ,n_{2}}$, when defined, are from the same equivalence class in $\mathcal{A}$. Hence the equivalence class of $[\psi _{m ,1}]$ is infinite.\medskip  

Now, let $m_{1} \neq m_{2} .$ Then $ \lnot ([\psi _{m_{1} ,1}]E_{B}[\psi _{m_{2} ,1}])$ since the values of $\psi _{m_{1} ,1}$ and $\psi \text{}_{m_{2} ,1}$, when defined, are from different equivalence classes in $\mathcal{A}$. 
\end{proof}

\begin{corollary}
Let $\ \mathcal{A} =(A ,E)$ be a computable equivalence structure. Let $C$ be a cohesive set. Then $\Pi _{C}\mathcal{A} \cong \mathcal{A}$ iff $\mathcal{A}$ has a bounded character or infinitely many equivalence classes.
\end{corollary}

It was shown in \cite{CCH} that the following model-theoretic result holds for the equivalence structures. Let the formulas $\gamma _{k}(x)$  state that the equivalence class of $x$ has at least $k$ elements, where $k \in \omega  -\{0\}$. Then the language of equivalence relations $\{E\}$ expanded with unary predicates $\{\gamma _{k} :k \geq 1\}$ has quantifier elimination; i.e., every first-order formula in the original language is logically equivalent to a quantifier-free formula in the new language. Since the formulas $\gamma _{k}(x)$ are $\Sigma _{1}^{0}$ formulas, it follows that a computable equivalence structure $\mathcal{A}$ and its cohesive powers satisfy the same first-order sentences. However, in some cases, the distinction can be made by using computable (infinitary) sentences.

\begin{corollary}
Let $\ \mathcal{A} =(A ,E)$  be a computable equivalence structure with unbounded character and no infinite equivalence classes. Let $C$ be a cohesive set. Then there is a computable infinitary $\Sigma _{3}$ sentence $\alpha $ such that $\Pi _{C}\mathcal{A} \models $ $\alpha $ and  $\mathcal{A} \models  \lnot \alpha  .$
\end{corollary}

\begin{proof}
 We have that  $\Pi _{C}\mathcal{A}$ has an infinite equivalence class, while $\mathcal{A}$ does not. Let $\alpha $ be the sentence saying that there is an infinite equivalence class:\bigskip

$ \exists x\bigwedge _{n \in \omega }$ $[ \exists y_{1} \cdots  \exists y_{n}$ $\bigwedge _{1 \leq i <j \leq n}(y_{i} \neq y_{j}$ $ \wedge xEy_{i})] .$ \bigskip

\noindent  Hence $\alpha $ is a computable $\Sigma _{3}$ sentence true in $\Pi _{C}\mathcal{A}$ but false in $\mathcal{A}$.  
\end{proof}

\section{
Injection structures
}
We will now study structures with a single unary function. An \emph{injection structure} $\mathcal{A} =(A ,f)$ consists of a non-empty set $A$ with an one-to-one function $f :A \rightarrow A$. Let $f^{0}(a) =_{def}a$. Given $a \in A$, the \emph{orbit of }$a$\emph{\ under }$f$ is defined as\begin{equation*}\mathcal{O}_{f} (a) =\{b \in A :( \exists n \in \omega )[f^{n}(a) =b \vee f^{n} (b) =a]\}\text{.}
\end{equation*}

\noindent We have that the size of orbit $\mathcal{O}_{f} (a)$ is $k \geq 1$ if and only if $f^{k} (a) =a$ and $( \forall t <k)[f^{t}(a) \neq a]$. Hence the property that $c a r d (\mathcal{O}_{f} (a)) =k$ is computable. 

By analogy with the character of an equivalence structure, we define the \emph{character} $\chi  (\mathcal{A})$ of an injection structure $\mathcal{A}$ as follows:
\begin{equation*}\chi  (\mathcal{A}) =\{ \langle k ,n \rangle  \in  \langle \omega  -\{0\} \rangle  \times  \langle \omega  -\{0\} \rangle  :\mathcal{A}\;\text{has at least $n$}\ \text{orbits of size $k$}\}\text{.}
\end{equation*}Hence $ \langle k ,n \rangle  \in \chi  (\mathcal{A})$ if and only if\medskip  

$( \exists x_{1})\cdots  ( \exists x_{n})\left (\bigwedge _{i =1}^{n}c a r d (\mathcal{O}_{f} (x_{i})) =k\; \wedge \;\bigwedge _{i \neq j}( \forall t <k)[f^{t}(x_{i}) \neq x_{j}]\right )$.\medskip

By $r a n (f)$ we denote the range of $f$, $r a n (f) =f(A)$.  An injection structure $(A ,f)$ may have two types of infinite orbits: $Z$-orbits, which are isomorphic to $(\mathbb{Z} ,S)$ and in which every element is in $r a n (f)$, and $\omega $-orbits, which are isomorphic to $(\omega  ,S)$ and have the form $\mathcal{O}_{f} (a) =\{f^{n} (a) :n \in \omega \}$ for some $a \notin r a n (f)$. Thus, injection structures are characterized by the number of orbits of size $k$ for each finite $k$, and by the number of orbits of types $Z$ and of type $\omega $.

For every computable injection structure $\mathcal{A} =(A ,f)$, we have the following arithmetic complexity of important relations:

\begin{enumerate}
\item [(a)] $\{(k ,a) :a \in  ra n (f^{k})\}$ is a $\Sigma _{1}^{0}$ set, 

\item [(b)] $\{(a ,k) :c a r d (\mathcal{O}_{f} (a)) \geq k\}$ is a $\Sigma _{1}^{0}$ set, 

\item [(c)] $\{a :\mathcal{O}_{f} (a)\;\text{is infinite}\}$ is a $\Pi _{1}^{0}$ set, 

\item [(d)] $\{a :\mathcal{O}_{f} (a)\; \text{has type  }Z\}$ is a $\Pi _{2}^{0}$ set, 

\item [(e)] $\;\{a :\mathcal{O}_{f} (a) \text{has type }\omega \}$ is a $\Sigma _{2}^{0}$ set, and 

\item [(f)] $\chi  (\mathcal{A})$ is a $\Sigma _{1}^{0}$ set.\end{enumerate}

It was shown in \cite{CHR} that for any c.e.\ character
$K$, there is a computable injection structure $\mathcal{A} =(\omega  ,f)$
with character
$K$
and any specified finite or countably infinite number of orbits of types
$\omega $
and $Z$ and for which $ran(f)$
is computable and
$\{a :\mathcal{O}_{f}(a)\;\text{is finite}\}$
is computable.

\begin{theorem}
\cite{CHR} A computable injection structure $\mathcal{A}$ is computably categorical if and only if $\mathcal{A}$ has finitely many infinite orbits.
\end{theorem}

Hence if $\mathcal{A}$ is an injection structure with finitely many infinite orbits and $\mathcal{D}$ is a computable structure isomorphic to $\mathcal{A} ,$ then for any cohesive set $C$ we have $\Pi _{C}\mathcal{A} \cong \Pi _{C}\mathcal{D} .$\smallskip 

By $f^{ -1}(b)$ we will denote the unique $a$ such that $f(a) =b$ if it exists, in symbols $f^{ -1}(b)\downarrow  =a$; otherwise $f^{ -1}(b)$ is not defined which we also denote by $f^{ -1}(b)\uparrow $. For $n \geq 1 ,$ we denote by $f^{ -n}$ a partial function $(f^{ -1})^{n}$. \medskip

Since the injection structures have the following axiom \medskip

$ \forall x \forall y[f(x) =f(y) \Rightarrow x =y]$, \medskip

\noindent we have that a cohesive power of a computable injection structure is an injection structure. We would like to determine the isomorphism types of such cohesive powers.
\begin{proposition}
 \label{basicinj}Let $\ \mathcal{A} =(A ,f)$ be a computable injection structure. Let $C$ be a cohesive set, and $\mathcal{B} =\Pi _{C}\mathcal{A} .$

(a) Then $\chi (\mathcal{B}) =\chi (\mathcal{A}) .$

(b) The structures $\mathcal{A}$ and $\mathcal{B}$ have the same number of $\omega $-orbits.

(c)\  If $\mathcal{A}$ has bounded character and no infinite orbits, then $\mathcal{A} \cong \mathcal{B} .$
\end{proposition}

\begin{proof}
(a) The character of a computable injection structure is definable by a $\Sigma _{1}^{0}$ sentence.\medskip

(b) Let $n \in \omega  -\{0\}.$ We can say that there are $ \geq n$  many $\omega $-orbits by the following $\Sigma _{2}^{0}$ sentence $\sigma _{n}$:\medskip

$( \exists x_{1}) \cdots ( \exists x_{n})( \forall y)[$
$\bigwedge _{1 \leq i <j \leq n}x_{i} \neq x_{j}$ $ \wedge $$\bigwedge _{1 \leq i \leq n}x_{i} \neq f(y)]$.

\medskip Thus, having $ <n$ many $\omega $-orbits can be expressed by a $\Pi _{2}^{0}$ sentence, $ \lnot \sigma _{n}$. Hence if $\mathcal{A}$ has no $\omega $-orbits, $\mathcal{B}$ has no $\omega $-orbits; and if $\mathcal{A}$ has exactly $n$ many $\omega $-orbits, $\mathcal{B}$ has exactly $n$ many $\omega $-orbits. If $\mathcal{A}$ has infinitely many $\omega $-orbits, then for every $n \geq 1 ,$ $\mathcal{A} \vDash \sigma _{n}$ and hence $\mathcal{B} \vDash \sigma _{n}$, so $\mathcal{B}$ has infinitely many $\omega $-orbits. The last conclusion also follows from the fact that $\mathcal{A}$ can be embedded into $\mathcal{B}$. \medskip

(c) Since the character of $\mathcal{A}$ is bounded, let $k \in \omega  -\{0\}$ be the largest size of a finite orbit of $\mathcal{A}$.

Since $\mathcal{A}$ has no infinite orbits, $\mathcal{A}$ satisfies the following $\Pi _{1}^{0}$ sentence, saying that there are no orbits of size $k +1$:\medskip 

$ \lnot ( \exists x)[\bigwedge _{1 \leq i \leq k +1}f^{i}(x) \neq x]$ \medskip

\noindent Thus, $\mathcal{B}$ satisfies the same sentence and, since $\mathcal{B}$ has the same character as $\mathcal{A} ,$ it is isomorphic to $\mathcal{A}$. 
\end{proof}

\begin{theorem}
Let $\ \mathcal{A}$ be a computable injection structure with unbounded character. Let $C$ be a cohesive set, and $\mathcal{B} =\Pi _{C}\mathcal{A} .$ Then $\mathcal{B}$  has infinitely many $Z$-orbits. 

Hence if $\mathcal{A}$ has unbounded character and finitely many $Z$-orbits, then $A \ncong \Pi _{C}\mathcal{A} .$
\end{theorem}

\begin{proof}
Without loss of generality, we will assume that $A =\omega  .$ Choose a computable sequence of elements in $A :\medskip $

$a(0) ,a(1) ,a(2) ,a(3) ,\ldots $ \medskip

\noindent such that they all belong to distinct finite orbits, and for every $k ,$ the orbit of $a(k)$ has $ >k$ elements. Since $f$ is computable and the character of $\mathcal{A}$ is unbounded, such a sequence can be obtained by simultaneously enumerating the elements of $A$ and checking the conditions.

For every natural number $m$ and an integer $z \in \mathbb{Z} ,$ we will define a computable function $\psi _{m ,z}$  as follows:\bigskip

$\psi _{m ,z}(i) =f^{z}(a( \langle m ,i \rangle )) .$\bigskip

Each $[\psi _{m ,z}]$ is in
$B .$ Fix $m$. Let integers $z_{1} ,z_{2}$ be such that $z_{1} \neq z_{2}$. Starting with some $i_{0} ,$ the orbit of $a( \langle m ,i \rangle )$  will be so large that for $i \geq i_{0}$, we will have \medskip

 $f^{z_{1}}(a( \langle m ,i \rangle )) \neq f^{z_{2}}(a( \langle m ,i \rangle ))$, \medskip  

\noindent hence $[\psi _{m ,z_{1}}] \neq [\psi _{m ,z_{2}}]$. In addition,  $[\psi _{m ,z_{1}}]$ and $[\psi _{m ,z_{2}}]$ belong to the same orbit in $\mathcal{B}$ since $f^{z_{2} -z_{1}}([\psi _{m ,z_{1}}]) =[\psi _{m ,z_{2}}]$. Hence the set $\{[\psi _{m ,z}] :z \in \mathbb{Z}\}$ forms a $Z$-orbit. On the other hand, if $m_{1} \neq m_{2} ,$ then for every $i ,$ $ \langle m_{1} ,i \rangle  \neq  \langle m_{2} ,i \rangle $, so $a( \langle m_{1} ,i \rangle )$ and $a( \langle m_{2} ,i \rangle )$ belong to different orbits. Hence $[\psi _{m_{1} ,0}]$ and $[\psi _{m_{2} ,0}]$ belong to different orbits in $\mathcal{B} ,$ so there are infinitely many $Z$-orbits.  
\end{proof}

\begin{theorem}
Let $\ \mathcal{A}$ be a computable injection structure with an infinite orbit. Let $C$ be a cohesive set, and $\mathcal{B} =\Pi _{C}\mathcal{A} .$ Then $\mathcal{B}$ has infinitely many $Z$-orbits.

Hence if $\mathcal{A}$ has an infinite orbit, but has at most finitely many $Z$-orbits, then $\mathcal{A} \ncong \Pi _{C}\mathcal{A} .$
\end{theorem}

\begin{proof}
  Assume that $\mathcal{A}$  has an infinite orbit. Let $a$ be an element of such an infinite orbit. For every natural number $m$ and every integer $z$, we define a partial computable function $\psi _{m ,z}$ by $\psi _{m .z}(n) =f^{m(2n +1) +n -1 +z}(a)$. Since the domain of each $\psi $ is co-finite, we have that $[\psi _{m ,z}] \in B$. For any $m$ and $z_{1} ,z_{2}$ such that $z_{1} \neq z_{2} ,$ we have that $\{n :$ $\psi _{m ,z_{1}}(n)\downarrow  =\psi _{m ,z_{2}}(n)\downarrow \} = \varnothing $ since the orbit of $a$ is infinite; hence $[\psi _{m ,z_{1}}] \neq [\psi _{m ,z_{2}}]$. We also have that $f([\psi _{m ,z}]) =[f \circ \psi _{m ,z}] =[\psi _{m ,z +1}]$. Hence for every $m$ we have a $Z$-orbit $\{[\psi _{m ,z}] :z \in \mathbb{Z}\} =\mathcal{O}([\psi _{m ,0}]) .$ 

Now, assume that $m_{1} \neq m_{2}$. Since for any $k \in \omega $, we have $f^{k}([\psi _{m_{1} ,0}]) \neq [\psi _{m_{2} ,0}]$, it follows that $\mathcal{O}([\psi _{m_{1} ,0}]) \neq \mathcal{O}([\psi _{m_{2} ,0}])$. Hence $\mathcal{B}$ has infinitely many $Z$-orbits. 
\end{proof}

\begin{corollary}
Let $\ \mathcal{A} =(A ,f)$ be a computable injection structure. Let $C$ be a cohesive set. Then $\Pi _{C}\mathcal{A} \cong \mathcal{A}$ iff $\mathcal{A}$ has a bounded character and no infinite orbits, or $\mathcal{A}$ has infinitely many $Z$-orbits.
\end{corollary}

It was shown in \cite{CHR} that the following model-theoretic result holds for the injection structures. Let the formulas $\gamma _{k}(x)$ state that $( \exists y)[f^{k}(y) =x]$. Then in the language of injection structures $\{f\}$ expanded with unary predicates $\{\gamma _{k} :k \geq 1\}$ we have quantifier elimination; i.e., every first-order formula in the original language is logically equivalent to a quantifier-free formula in the new language. Since the formulas $\gamma _{k}(x)$ are $\Sigma _{1}^{0}$ formulas, it follows that a computable injection structure $\mathcal{A}$ and its cohesive powers satisfy the same first-order sentences. However, in some cases, the distinction can be made by using computable (infinitary) sentences.

\begin{corollary}
If $\mathcal{A}$ is a computable injection structure with an unbounded character and no infinite orbits, then there is a computable (infinitary) $\Sigma _{2}$ sentence $\alpha $ such that $\mathcal{B} \vDash \alpha $ and $\mathcal{A} \nvDash \alpha $.
\end{corollary}

\begin{proof}
 Let $\alpha $ say that there is an infinite orbit: \bigskip

$ \exists x\bigwedge _{k \in \omega }[f^{(k)}(x) \neq x]$.  
\end{proof}

\section{Two-to-one structure}
We will now investigate cohesive powers of two-to-one structures that were introduced and studied in \cite{CHR1} from the computability-theoretic point of view with focus on the complexity of isomorphisms
between these structures.

\begin{definition}
 A \emph{two-to-one structure} $\mathcal{A} =(A ,f)$ consists of a non-empty domain $A$ with a single unary function $f :A \rightarrow A$ such that for every $a \in A$ we have $card(f^{ -1}(a)) =2$.
\end{definition}

We will also call a two-to-one structure a 2:1 structure, and often identify it with its directed graph $G_{\mathcal{A}}$ with vertex set $A$ and edges $(a ,f(a))$ for $a \in A$.

\begin{definition}
For $a \in A ,$ the orbit of $a$ is:\bigskip

$\mathcal{O}_{f}(a) =\{x \in A :( \exists n ,m \in \omega )\left .[f^{n}(x) =f^{m}(a)]\right .\} .$

\end{definition}

\noindent That is, the orbit of $a$ is the set of elements of $A$, which belong to the same connected component of $G_{\mathcal{A}}$ to which $a$ belongs.  

Let $B$ be a full binary tree with its nodes pointing toward the root. We can show that there are two types of orbits in a 2:1 structure: $Z$-chains, and $k$-cycles for $k \geq 1$. A $Z$-chain consists of a directed one-to-one basic sequence of nodes ordered as integers, with a tree $B$ attached to every node of the basic sequence as follows: there is a connecting edge, which is not part of the basic sequence but points toward the sequence, to which the root of $B$ is attached. A $k$-cycle is a directed one-to-one cycle of size $k$ such that a tree $B$ is attached to each node of the cycle via a connecting edge, which is not part of the cycle but points toward the cycle, to which the root of $B$ is attached. Hence all tree edges point toward the cycle. For pictures illustrating orbits see Section 1 in \cite{CHR1}.

\begin{lemma}
\cite{CHR1} (i) The predicate ``$O_{f}(a)$ is a $k$-cycle'' is $\Sigma _{1}^{0} .$ \medskip

(ii) The predicate ``$O_{f}(a)$ is a $Z$-chain'' is $\Pi _{1}^{0}$.
\end{lemma}

Two countable 2:1 structures are isomorphic if they have the same number of $k$-cycles for every $k \geq 1$,  and the same number of $Z$-chains.\medskip

2:1 structures have the following axioms:\medskip

$ \forall y \exists x_{1} \exists x_{2}[x_{1} \neq x_{2} \wedge f(x_{1}) =y \wedge f(x_{2}) \Rightarrow y]$ \medskip

$ \forall x_{1} \forall x_{2} \forall x_{3}[(f(x_{1}) =f(x_{2}) \wedge f(x_{2}) =f(x_{3})) \Rightarrow (x_{1} =x_{2} \vee x_{1} =x_{3} \vee x_{2} =x_{3})]$ \medskip

\noindent Hence a cohesive power of a computable 2:1 structure is a 2:1 structure. We would like to determine the isomorphism types of these cohesive powers.

\begin{theorem}
\label{2-1Char}Let $\ \mathcal{A}$ be a computable \emph{$2 :1$} structure. Let $C$ be a cohesive set, and $\mathcal{B} =\Pi _{C}\mathcal{A} .$  \medskip  

(i) The cohesive power $\mathcal{B}$ has the same number of $k$-cycles, for any $k \geq 1$, as $\mathcal{A}$ does. \medskip

(ii) The cohesive power $\mathcal{B}$ has infinitely many $Z$-chains.\medskip  

Hence if $\mathcal{A}$ has at most finitely many $Z$-chains, then $\mathcal{A} \ncong \Pi _{C}\mathcal{A} .$
\end{theorem}

\begin{proof}
 (i) The property that a 2:1 structure has at least $n$ many $k$-cycles, where $n ,k \geq 1$ can be expressed by an existential sentence $\theta _{n ,k}$: \bigskip    

$\noindent ( \exists x_{1}) \cdots ( \exists x_{n})[\bigwedge _{1 \leq m \leq n}(f^{k}(x_{m}) =x_{m} \wedge (\bigwedge _{1 \leq l <k}f^{l}(x_{m}) \neq x_{m}))$\bigskip 

$ \wedge \bigwedge _{(1 \leq i <j \leq n) \&(1 \leq l <k)}f^l(x_{i}) \neq x_{j}]$.\bigskip 

\noindent Hence both $\mathcal{A}$ and its cohesive power $\mathcal{B}$ satisfy the same such sentences, so they have the same number of $k$-cycles.\medskip

(ii)\  Fix a natural ordering on the domain $A .$ We will ``abuse'' the notation and by $f^{ -1}$ denote the unary function on $A ,$ which for every $a$ chooses the smaller of the two elements that $f$ maps into $a$. Hence $f^{ -z}$ will be defined for every integer $z$ where, as usual, $f^{0}(a) =a .$ 

Since $\mathcal{A}$ always contains a full binary tree component $\mathcal{T}$, we can define a computable function $g :\omega  \rightarrow A ,$ which chooses elements $g(n)$ on $\mathcal{T}$ that are spaced apart so that $f^{z}(g(n))$ where $\left \vert z\right \vert  \leq n$ do not ``interfere'' for different $n$'s.  More precisely, if $n_{1} \neq n_{2}$ or $z_{1} \neq z_{2} ,$ then $f^{z_{1}}(g(n_{1})) \neq f^{z_{2}}(g(n_{2}))$ where $\left \vert z_{1}\right \vert  \leq n_{1}$ and $\left \vert z_{2}\right \vert  \leq n_{2}$. Equivalently,\medskip  

$\left \vert f^{z}(g(n)) : -n \leq z \leq n \wedge 0 \leq n \leq m\}\right \vert  =(m +1)^{2} $.\medskip  

\noindent Our goal is to use this property to define partial computable functions $\psi _{m ,z}$ for natural numbers $m$ and integers $z$, such that $\psi _{m , \ast }$'s witness that there are infinitely many $Z$-chains.  

A partial function $\psi _{m ,z} :\omega  \rightarrow A$ is defined as follows:\bigskip

$\psi _{m ,z}(x) =\left \{\begin{array}{cc}f^{z}(g( \langle m ,x \rangle )) & \text{if }\left \vert z\right \vert  \leq  \langle m ,x \rangle  ; \\
\uparrow  & \text{otherwise.}\end{array}\right .$

\bigskip \noindent It follows that $[\psi _{m ,z}] \in B$ for every $m ,z$. Furthermore, $f([\psi _{m ,z}]) =[\psi _{m ,z +1}]$, so $\{[\psi _{m ,z}] :z \in \mathbb{Z}\}$ is a subset of a $Z$-chain. For  any pair of natural numbers $m_{1}$ and $m_{2}$ such that $m_{1} \neq m_{2}$ and arbitrary $k_{1} ,k_{2}$, we have $f^{k_{1}}([\psi _{m_{1} ,0}]) \neq f^{k_{2}}([\psi _{m_{2} ,0}])$, so $[\psi _{m_{1} ,0}]$ and $[\psi _{m_{2} ,0}]$ belong to different $Z$-chains. Hence $\mathcal{B}$ has infinitely many $Z$-chains.
\end{proof}

\begin{corollary}
Let $\ \mathcal{A}$ be a computable \emph{$2 :1$} structure. Let $C$ be a cohesive set.
Then $\mathcal{A} \cong \Pi _{C}\mathcal{A}$ iff $A$ has infinitely many $Z$-chains.

\end{corollary}

\begin{corollary}
Let $\mathcal{A}$ be a computable $2 :1$ structure with no $Z$-chains. Let $C$ be a cohesive set, and $\mathcal{B} =\prod _{C}\mathcal{A}$. Then there is a computable (infinitary) $\Sigma _{2}$ sentence $\alpha $ such that $\mathcal{B} \vDash \alpha $ and $\mathcal{A} \nvDash \alpha $. 
\end{corollary}

\begin{proof}
  Let $\alpha $ say that there is a $Z$-chain: \bigskip

$ \exists x\bigwedge _{l \geq 0 \&k >0}(f^{(l +k)}(x) \neq f^{l}(x))$. 
\end{proof}

\section{
(2,0):1 structures}
We will now investigate the class of (2,0):1 structures, which includes 2:1 structures. They were introduced and studied in \cite{CHR1} from the computability-theoretic point of view with focus on the complexity of isomorphisms
between these structures.

\begin{definition}
A $(2 ,0) :1$ \emph{structure} is a structure with a single unary function, $\mathcal{A} =(A ,f)$ where $f :A \rightarrow A$, such that for every $a \in A$, we have  $c a r d (f^{ -1} (a)) \in \{0 ,2\}$. 
\end{definition}

As usual, a (2,0):1 structure $\mathcal{A}$ is often identified with its directed graph $G (A ,f)$,  and the orbit of $a$ is defined to be the set of all points in $A$ which belong to the connected component of $G (A ,f)$ containing $a$. The orbits of (2,0):1 structures can be $k$-cycles for $k \geq 1 ,$ $Z$-chains, or $\omega $-chains. A $k$-cycle consists of a directed one-to-one cycle of size $k$ such that for each node of the cycle there is a connecting edge, which is not part of the cycle and pointing toward the cycle, to which the root of a binary tree is attached with all tree edges pointing toward the cycle. Here, a tree can be finite or infinite and it has to satisfy the condition  $c a r d (f^{ -1} (a)) \in \{0 ,2\}$. \medskip

A $Z$-chain consists of a directed one-to-one basic sequence of nodes ordered as integers, with a binary tree attached to every node of the basic sequence as follows: there is a connecting edge, which is not part of the basic sequence but points toward the sequence, to which the root of a tree is attached. Hence each element of a $k$-cycle or a $Z$-chain also has a binary branching tree attached to it as its root and with all edges directed toward the root. An $\omega $-chain consists of a directed basic sequence of nodes ordered as natural numbers such that for every node except the first one there is a binary branching tree attached as above, with connecting edge pointing toward the basic sequence and the tree edges pointing toward the root. Hence the first node of the basic sequence does not belong to the range of $f$. For pictures illustrating orbits see Section 1 in \cite{CHR1}. \medskip

Let $(A ,f)$ be a (2,0):1 structure and $a \in A$. The \emph{length of} $a$ is defined as: \medskip

$l(a) =\sup \{n +1 :\left \vert \{a ,f(a) ,f^{2}(a) ,\ldots  ,f^{n +1}(a)\}\right \vert $$ =n +1\}$. \medskip

\noindent  It is the longest ``non-cycling'' directed path starting with $a .$ It can be finite or infinite.

\begin{definition}
Let $\mathcal{A} =(A ,f)$ be a $(2 ,0) :1$ structure. Let $k ,n \in \omega  -\{0\}$.  \medskip  

\noindent (a) The \emph{cycle character} of $\mathcal{A}$ is

\begin{center}
$\chi _{cycle}(\mathcal{A}) =\{ \langle k ,n \rangle  :$ $\mathcal{A}$ has $ \geqslant n$ many $k$-cycles$\}$. \bigskip 
\end{center}\par
\noindent (b) The \emph{path character} of $\mathcal{A}$ is \medskip

\begin{center}
$\chi _{path}(\mathcal{A}) =\{ \langle k ,n \rangle  :$ $\mathcal{A}$ has $ \geqslant n$ many $a$ such that $l(a) =k\}$. \bigskip 
\end{center}\par
\noindent (c) The \emph{endpath character} of $\mathcal{A}$ is

\begin{center}
$\chi _{endpath}(\mathcal{A}) =\{ \langle k ,n \rangle  :$ $\mathcal{A}$ has $ \geqslant n$ many $a \notin f(A)$ such that $l(a) =k\}$. \bigskip 
\end{center}\par
\end{definition}

\noindent
We say that a character is bounded if there is an upper bound on $k$. \bigskip

(2,0):1 structures have the following axioms: \bigskip

$ \forall y \forall x \exists z[f(x) \neq y \vee (z \neq x \wedge f(z) =y)]$  and\bigskip

$ \forall x_{1} \forall x_{2} \forall x_{3}[(f(x_{1}) =f(x_{2}) \wedge f(x_{2}) =f(x_{3})) \Rightarrow (x_{1} =x_{2} \vee x_{1} =x_{3} \vee x_{2} =x_{3})]$.\bigskip

\noindent Hence a cohesive power of a computable (2,0):1 structure is a (2,0):1 structure.

\begin{proposition}
Let $\mathcal{A}$ be a computable $(2 ,0) :1$  structure. Let $C$ be a cohesive set, and let $\mathcal{B} =\prod _{C}\mathcal{A}$. Then $\mathcal{A}$ and $\mathcal{B}$ have tha same cycle character, path character, and endpath character.
\end{proposition}

\begin{proof}
 Let $k ,n \geq 1$. Then $ \langle k ,n \rangle  \in \chi _{cycle}(\mathcal{A})$ can be expressed by a $\Sigma _{1}^{0}$ sentence as in the proof of Theorem \ref{2-1Char} (i). \medskip  

Furthermore, $ \langle k ,n \rangle  \in \chi _{path}(\mathcal{A})$ can be expressed  by the following $\Sigma _{1}^{0}$ sentence:

$\noindent ( \exists x_{1}) \cdots ( \exists x_{n})[\bigwedge _{1 \leq m \leq n}((\bigwedge _{0 \leq l <s <k}f^{l}(x_{m}) \neq f^{s}(x_{m})) \wedge $ \bigskip

$\bigvee _{1 \leq l <k}(f^{k}(x_{m}) =f^{l}(x_{m})))$$ \wedge \bigwedge _{1 \leq i <j \leq n}x_{i} \neq x_{j}]$.\bigskip

Finally, $ \langle k ,n \rangle  \in \chi _{endpath}(\mathcal{A})$ can be expressed  by the following $\Sigma _{2}^{0}$ sentence:\bigskip 

$( \exists x_{1}) \cdots ( \exists x_{n})[\bigwedge _{1 \leq m \leq n}((\bigwedge _{0 \leq l <s <k}f^{l}(x_{m}) \neq f^{s}(x_{m})) \wedge $ \bigskip

$\bigvee _{1 \leq l <k}(f^{k}(x_{m}) =f^{l}(x_{m})) \wedge ( \forall y)(f(x_{m}) \neq y)) \wedge \bigwedge _{1 \leq i <j \leq n}x_{i} \neq x_{j}]$.\bigskip

By Theorem \ref{rumen} it follows that $\mathcal{A}$ and $\mathcal{B}$ satisfy the same sentences above for every pair $(k ,n) ,$ so $\mathcal{A}$ and $\mathcal{B}$ have the same characters.
\end{proof}

\begin{theorem}
Let $C$ be a cohesive set. \medskip

(i) Let $\mathcal{A}$ be a computable $(2 ,0) :1$ structure with bounded path character and no infinite orbits. Then $\Pi _{C}\mathcal{A} \cong \mathcal{A}$. \medskip

(ii) Let $\mathcal{A}$ be a computable $(2 ,0) :1$ structure with unbounded path character or with an infinite orbit. Then $\Pi _{C}\mathcal{A}$ has infinitely many $Z$-chains. \medskip

Hence if
$\mathcal{A}$ has only finitely many (including zero) $Z$-chains, we have $\mathcal{A} \ncong \Pi _{C}\mathcal{A}$.
\end{theorem}

\begin{proof}
 (i) Let $\mathcal{B} =\Pi _{C}\mathcal{A}$. Let $M \in \omega $ be the least upper bound for the length of (finite) paths in $\mathcal{A}$. Since $\mathcal{A}$ does not have infinite orbits, it satisfies the following $\Pi _{1}^{0}$ sentence: \bigskip

$( \forall x)\bigvee _{0 \leq m <n \leq M}(f^{m}(x) =f^{n}(x))$.\bigskip

\noindent By Theorem \ref{rumen}, $\mathcal{B}$ also satisfies this sentence, which implies that $\mathcal{B}$ has no infinite orbits. Since we can completely describe finite orbits by $\Sigma _{2}^{0}$ sentences, it follows that $\mathcal{A}$ and $\mathcal{B}$ are isomorphic.\bigskip

(ii) Fix a natural ordering on the domain $A .$ We will denote by $f^{ -1}$ a partially computable unary function on $A ,$ which for every $a$ chooses the smaller of the two elements that $f$ maps into $a ,$ if $c a r d (f^{ -1} (a)) =2$, and is undefined otherwise.  Hence we have $f^{ -z}$ for every integer $z$, where $f^{0}(a) =a$. \medskip

By the assumption about $\mathcal{A} ,$ we can define a computable function $g :\omega  \rightarrow A ,$ which chooses elements $g(n)$ in $A$ such that for every $m \in \omega  ,$ we have $\left \vert \{f^{z}(g(n)) :f^{z}(g(n))\downarrow  \wedge \ 0 \leq n \leq m \wedge  -n \leq z \leq m\}\right \vert  =(m +1)^{2}$. We now proceed similarly as in (ii)  in the proof of Theorem \ref{2-1Char} to show that there are infinitely many $Z$-chains.
\end{proof}

\bigskip The following theorem focuses on $\omega $-chains.

\begin{theorem}
Let $C$ be a cohesive set. Let $\mathcal{A}$ be a computable $(2 ,0) :1$ structure with bounded endpath character and with finitely many (including $0$) elements in $A -f(A)$ of infinite length. Then $\prod _{C}\mathcal{A}$ and $\mathcal{A}$ have the same number of $\omega $-chains.

\end{theorem}

\begin{proof}
Let $\mathcal{B} =\Pi _{C}\mathcal{A}$. Let $M \in \omega $ be the least upper bound for the endpath character of $\mathcal{A}$; that is, $M =\max \{k : \langle k ,n \rangle  \in \chi _{endpath}(\mathcal{A})\}$.

Suppose that in $\mathcal{B}$ we choose $[\psi ] \in B -f(B)$ such that $[\psi ]$ is an element of an $\omega $-chain. We will show that $[\psi ]$ belongs to the range of the canonical embedding function of $\mathcal{A}$ info $\mathcal{B}$.  

Consider a c.e.\ set $W =\{i \in \omega  :\psi (i)\uparrow  \vee \ l(\psi (i)) >M\}$, which is infinite. We claim that $C \subseteq ^{ \ast }W$. To establish the claim, assume otherwise, hence $C \subseteq ^{ \ast }\overline{W}$. We have that $\overline{W} =\{i \in \omega  :\psi (i)\downarrow  \wedge \ \ l(\psi (i)) \leq M\}$ so it is a finite union of c.e.\ sets $Y_{j}$ for $j \in \{1 ,\ldots  ,M\}$ where\bigskip

$Y_{j} =\bigcup _{0 \leq k <j}\{i \in \omega  :\psi (i)\downarrow  \wedge f^{j}(\psi (i))\downarrow  =f^{k}(\psi (i))\downarrow  \wedge $ \bigskip

$\bigwedge _{0 \leq m <n <j}f^{m}(\psi (i))\downarrow  \neq f^{n}(\psi (i))\downarrow \}$. \bigskip 

\noindent Hence, since $C$ is cohesive, for some $j_{0} ,k_{0}$ such that $j_{0} >k_{0}$ we have that\bigskip  

$C \subseteq ^{ \ast }\{i \in \omega  :\psi (i)\downarrow  \wedge f^{j_{0}}(\psi (i))\downarrow  =f^{k_{0}}(\psi (i))\downarrow  \wedge $ \bigskip

$\bigwedge _{0 \leq m <n <j_{0}}f^{m}(\psi (i))\downarrow  \neq f^{n}(\psi (i))\downarrow \}$,\bigskip 

\noindent which will imply that $f^{j_{0}}[\psi ] =f^{k_{0}}[\psi ]$, contradicting the fact that $[\psi ]$ is an element of an $\omega $-chain. Hence \medskip

 $C \subseteq ^{ \ast }W =\{i \in \omega  :\psi (i)\uparrow  \vee \ l(\psi (i)) >M\}$. \medskip  

\noindent Since $C \subseteq ^{ \ast }\{i \in \omega  :\psi (i)\downarrow \} ,$ we have \medskip  

$C \subseteq ^{ \ast }\{i \in \omega  :\psi (i)\downarrow  \wedge \ l(\psi (i)) >M\} =$ \medskip

$\{i \in \omega  :\psi (i)\downarrow  \wedge \ l(\psi (i)) >M \wedge \psi (i) \in f(A)\}\  \cup $ \smallskip

 $\{i \in \omega  :\psi (i)\downarrow  \wedge \ l(\psi (i)) >M \wedge f(i) \in A -f(A)\} .$ \bigskip  

\noindent We can further show that $C \subseteq ^{ \ast }\{i \in \omega  :\psi (i)\downarrow  \wedge \ \ l(\psi (i)) >M \wedge f(i) \in A -f(A)\}$ since, otherwise, there is $[\tau ] \in B$ such that $f([\tau ]) =[\psi ]$, contradicting the fact that $[\psi ] \in B -f(B) .$

Hence, $C \subseteq ^{ \ast }\bigcup \{X_{a} :a \in (A -f(A)) \wedge l(a) =\infty \} .$ By assumption, the union in the previous formula is a finite union, so for some $a_{0} ,$ we have\medskip 

 $C \subseteq \{i \in \omega  :\psi (i)\downarrow  =a_{0}\}$. \medskip  

\noindent Thus, $[\psi ]$ belongs to the range of the canonical embedding. It follows that the number of $\omega $-chains is the same in $\mathcal{A}$ and $\mathcal{B}$.
\end{proof}

\bigskip For the following result we require more decidability in a computable structure.

\begin{definition}
\cite{CHR1} A computable $(2 ,0) :1$ structure $\mathcal{A} =(A ,f)$ is said to be \emph{highly computable} if $ran(f)$ is a computable set.
\end{definition}

\begin{theorem}
\label{omegachains}Let $C$ be a cohesive set. Let $\mathcal{A}$ be a highly computable $(2 ,0) :1$ structure with unbounded endpath character or with infinitely many elements in $A -f(A)$ of infinite length. Then $\Pi _{C}\mathcal{A}$ has infinitely many $\omega $-chains.\medskip  

Hence if $\mathcal{A}$ has only finitely many $\omega $-chains, we have $\mathcal{A} \ncong \Pi _{C}\mathcal{A} .$
\end{theorem}

\begin{proof}
Let $B =\Pi _{C}\mathcal{A}$. Let $g :\omega  \rightarrow A$ be a computable function such that $g(\omega ) \subseteq A -f(A)$ and for every $m \in \omega $  we have \medskip

 $\vert \{f^{k}(g(n)) :f^{k}(g(n))\downarrow  \wedge \ 0 \leq n \leq m \wedge \ 0 \leq k \leq m\} =\frac{(m +1)(m +2)}{2}$. \medskip  

A partial function $\psi _{m ,n} :\omega  \rightarrow A$ for $m ,n \in \omega $ is defined as follows:\bigskip

$\psi _{m ,n}(x) =\left \{\begin{array}{cc}f^{n}(g( \langle m ,x \rangle )) & \text{if }0 \leq n \leq  \langle m ,x \rangle  ; \\
\uparrow  & \text{otherwise.}\end{array}\right .$

\bigskip \noindent It follows that every $[\psi _{m ,n}] \in B$. If $n_{1} \neq n_{2} ,$ then $[(\psi _{m ,n_{1}})] \neq [(\psi _{m ,n_{2}})]$. Furthermore, $[(\psi _{m ,0})] \in B -f(B)$ and $f([\psi _{m ,n}]) =[\psi _{m ,n +1}]$ , so $\{[\psi _{m ,n}] :n \in \omega \}$ is a subset of an $\omega $-chain. For  any pair of natural numbers $m_{1}$$ ,m_{2}$ such that $m_{1} \neq m_{2}$ and arbitrary $k_{1} ,k_{2}$, we have $f^{k_{1}}([\psi _{m_{1} ,0}]) \neq f^{k_{2}}([\psi _{m_{2} ,0}])$, so $[\psi _{m_{1} ,0}]$ and $[\psi _{m_{2} ,0}]$ belong to different $\omega $-chains. Hence $\mathcal{B}$ has infinitely many $\omega $-chains.
\end{proof}

\begin{corollary}
If $\mathcal{A}$ is a highly computable $(2 ,0) :1$ structure with unbounded endpath character and no $\omega $-chains. Let $C$ be a cohesive set, and  $\mathcal{B} =\prod _{C}\mathcal{A}$.  Then there is a computable $\Sigma _{2}$ sentence $\alpha $ such that $\mathcal{B} \vDash \alpha $ and $\mathcal{A} \nvDash \alpha $.

\end{corollary}

\begin{proof}
The property that there is an $\omega $-chain can be expressed by a computable (infinitary) $\Sigma _{2}$ sentence $\alpha $:\medskip

\begin{center}
$ \exists x(\bigwedge _{y \in A}(x \neq f(y)) \wedge \bigwedge _{l \geq 0 \wedge k >0}(f^{(l +k)}(x) \neq f^{n}(x)))$ 
\end{center}\par
\end{proof}

\section{
Partial injection structures}
A \emph{partial injection structure} $(A ,f)$ consists of a set $A$ and a partial function $f :A \rightarrow A$ such that if $x ,y \in dom(f)$ and $x \neq y$, then $f(x) \neq f(y) .$ We will call $f$ a \emph{partial injection}. As usual, we write $f(x)\downarrow $ to denote that $x \in dom(f) ,$ and $f(x)\uparrow $ to denote that $x \notin dom(f)$. Also, $f(x)\downarrow y$ stands for $f(x)\downarrow $ and $f(x) =y$. Partial inverse function $f^{ -1}$ is defined naturally.  For $z \in \mathbb{Z}$,  $f^{z}$ is defined as the usual composition of partial functions. Partial injection structures and their computability-theoretic properties, including complexity of their isomorphisms, were studied by Marshall in \cite{LM}. She calls a partial injection structure $(A ,f)$ a \emph{partial computable injection structure} if $A$ a computable set and $f$  is a partial computable function. 

In order to make a partial injection structure $(A ,f)$ into a first-order structure, we will consider it as a relational structure $\mathcal{A} =(A ,G_{f})$, where $G_{f}$ is the graph of $f :$ \medskip

$G_{f} =\{(x ,y) :x \in dom(f) \wedge f(x) =y\} .$ \medskip

\noindent Having this framework in mind, we can still write $(A ,f) .$
\begin{definition}
We say that a partial computable injection structure $(A ,f)$ is a \emph{computable partial injection structure} if $G_{f}$ is a computable binary relation. Hence $(A ,G_{f})$ is a computable structure.

\begin{proposition}
Let $(A ,f)$ be a partial computable injection structure.

(i) If $ran(f)$ is computable, then $G_{f}$ is computable.

(ii) If $dom(f)$  is computable, then $G_{f}$ is computable.

\end{proposition}

\end{definition}

\begin{proof}
(i)               Without loss of generality, we may assume that $A$$ =\omega $. Given a pair $(x ,y) ,$ first determine whether $y \in ran(f) .$ If $y \notin ran(f)$, then $(x ,y) \notin G_{f} .$ If $y \in ran(f) ,$ then run a Turing machine program $P_{f}$ for computing $f$ simultaneously on $0 ,1 ,2 ,\ldots $ by adding more and more inputs and computation steps (although finitely many at every stage) until we find $z$ such that the program $P_{f}$ halts on $z$ and outputs $y$.  If $x =z$ then $(x ,y) \in G_{f} ,$ and if $x \neq z$ then $(x ,y) \notin G_{f}$. \medskip

(ii) Given a pair $(x ,y) ,$ first determine whether $x \in dom(f) ,$ and if that is the case compute $f(x) .$    
\end{proof}

\bigskip

The domain of a partial computable injection function with computable range does not have to be computable. For example, for the halting set $K ,$ consider a computable $1 -1$ enumeration $g :$ $\omega  \rightarrow K$.
Let $f :\omega  \rightarrow \omega $ be defined as:

$f(x) =\left \{\begin{array}{cc}g^{ -1}(x) & \text{if  }x \in K \\
\uparrow  & \text{otherwise}\end{array}\right .$

\medskip

\noindent Then $f$  is a partial computable injection with $ran(f) =\omega $ and $dom(f) =K$. \medskip

Let $(A ,f)$ be a partial injection structure. The orbit of $a$ is defined to be: \bigskip

$\mathcal{O}_{f}(a) =\{b :$ $ \exists n \in \omega (f^{n}(a)\downarrow  =b \vee f^{n}(b)\downarrow  =a)\}$. \bigskip

\noindent There are five kinds of orbits. Finite orbits may be $k$-cycles or $k$-chains for $k \geq 1.$  A $k$-chain is of the form $\{x_{i} :1 \leq i \leq k\}$ where $x_{i} \neq x_{j}$ for $1 \leq i <j \leq k$ and $x_{1} \in A -ran(f)$, $x_{i +1} =f(x_{i})$, and $x_{k} \notin dom(f$$)$. Infinite orbits may be $Z$-chains, $\omega $-chains, or $\omega ^{ \ast }$-chains.  An $\omega ^{ \ast }$-chain is of the form $\{x_{i} :i \in \omega \}$ where  $x_{0} \in A -dom(f)$, $x_{i} =f(x_{i +1})$, and $x_{i} \neq x_{j}$ for $i \neq j .$

\begin{definition}
Let $\mathcal{A} =(A ,f)$ be a partial injection structure. In the following definitions we will assume that $k ,n \in \omega  -\{0\}.$ 
\end{definition}

(a) The \emph{cycle character} of $\mathcal{A}$ is

$\chi _{cycle}(\mathcal{A}) =\{ \langle k ,n \rangle  :$ $\mathcal{A}$ has $ \geqslant n$ many $k$ -cycles$\} .$ \medskip

(b) The \emph{finite chain character} of $\mathcal{A}$ is

$\chi _{path}(\mathcal{A}) =\{ \langle k ,n \rangle  :$ $\mathcal{A}$ has $ \geqslant n$ many $k$ -chains$\}$. \medskip

We say that a character in the previous definition is bounded if there is an upper bound on the size $k$. Two countable partial injection structures are isomorphic if and only if they have the same cycle character, the same finite chain character nd the same number of $Z$-chains, $\omega $-chains and $\omega ^{ \ast }$-chains. \medskip  

Let $\mathcal{A} =(A ,f)$ be a computable partial injection structure. Let $C$ be a cohesive set, and $\mathcal{B} =\Pi _{C}\mathcal{A}$. Then $f^{\mathcal{B}}([\psi ])\downarrow $ if $C \subseteq ^{ \ast }\{i \in \omega  :f(\psi (i))\downarrow \}$ and  $f^{\mathcal{B}}([\psi ])\uparrow $ otherwise. Similarly, $f^{\mathcal{B}}([\psi ]) =[\phi ]$ if and only if $C \subseteq ^{ \ast }\{i \in \omega  :f(\psi (i))\downarrow  =\phi (i)\downarrow \}$. We often omit the superscript in $f^{\mathcal{B}}$. \medskip

The following proposition is based on Theorem \ref{rumen}.

\begin{theorem}
\label{partialchar}Let $\mathcal{A} =(A ,f)$ be a computable partial injection structure. Let $C$ be a cohesive set. Then the cohesive power $\Pi _{C}\mathcal{A}$ is a partial injection structure that has the same cycle character and finite chain character as $\mathcal{A}$.
\end{theorem}

\begin{proof}
Being a partial injection structure can be described by the following $\Pi _{1}^{0}$ sentences: \medskip

$ \forall x \forall y \forall z[(f(x) =z \wedge f(y) =z) \Rightarrow x =y]$ \bigskip

$ \forall x \forall y \forall z[(f(x) =y \wedge f(x) =z) \Rightarrow y =z]$ \bigskip

\noindent Furthermore, \bigskip

$ \langle k ,n \rangle  \in \chi _{cycle}(\mathcal{A})$ iff \medskip

$\noindent  \exists x_{1} \cdots  \exists x_{n}[$ $\bigwedge _{1 \leq i \leq n}(f^{k}($$x_{i}) =x_{i} \wedge \bigwedge _{1 \leq l <k}f^{l}(x_{i}) \neq x_{i}) \wedge \bigwedge _{1 \leq i <j \leq n \&1 \leq l <k}f^{l}(x_{i}) \neq x_{j}]$, \bigskip

\noindent which is a $\Sigma _{1}^{0}$ sentence. \bigskip

\noindent Also,

$ \langle k ,n \rangle  \in \chi _{path}(\mathcal{A})$ iff \medskip

$\noindent  \exists x_{1} \cdots  \exists x_{n} \exists y_{1} \cdots  \exists y_{k} \forall y[$ $\bigwedge _{1 \leq i <j \leq n}x_{i} \neq x_{j} \wedge \bigwedge _{1 \leq i \leq n}(f^{}(y) \neq x_{i} \wedge f^{k -1}(x_{i}) =y_{i} \wedge f(y_{i}) \neq y)]$, \bigskip

\noindent which is a $\Sigma _{2}^{0}$ sentence. 
\end{proof}

\bigskip 

In some cases, the cohesive power is isomorphic to the original structure.

\begin{theorem}
Let $\mathcal{A}$ be a computable partial injection structure with bounded cycle character, bounded finite chain character, and no infinite orbits.  Let $C$ be a cohesive set. Then $\mathcal{A} \cong \Pi _{C}\mathcal{A}$.
\end{theorem}

\begin{proof}
Let $\mathcal{A} =(A ,f)$ and $\mathcal{B} =\Pi _{C}\mathcal{A}$. Let $M$ be the maximum size of finite orbits. Since $\mathcal{A}$ has no infinite orbits, it satisfies the following $\Pi _{2}^{0}$ sentence: \bigskip

$ \lnot  \exists x \exists y(f^{M +1}(x) =y \wedge \bigwedge _{1 \leq i \leq M}$ $(f^{i}(x) \neq x)) .$ \bigskip

\noindent Hence $\mathcal{B}$ satisfies the same sentence, so it has no infinite orbits. Thus, together with Theorem \ref{partialchar}, we have $\mathcal{A} \cong \mathcal{B}$.
\end{proof}

\bigskip

If the conditions of the previous theorem are not satisfied, the cohesive power will have infinitely many $Z$-chains.

\begin{theorem}
Let $\mathcal{A}$ be a computable partial injection structure with unbounded cycle character, or with unbounded finite chain character, or with an infinite orbit.  Let $C$ be a cohesive set. Then $\Pi _{C}\mathcal{A}$ has infinitely many $Z$-chains. 

Hence if $\mathcal{A}$ has only finitely many $Z$-chains, then  $\mathcal{A} \ncong \Pi _{C}\mathcal{A} .$
\end{theorem}

\begin{proof}
The proof is similar to that of part (ii) of Theorem \ref{2-1Char}.
\end{proof}

\bigskip

The cohesive power always has at least as many orbits of a certain fixed type as $\mathcal{A}$. In some cases the number of $\omega $-chains and $\omega ^{ \ast }$-chans is the same.

\begin{theorem}
Let $\mathcal{A}$ be a computable partial injection structure with bounded finite chain character. Let $C$ be a cohesive set. Then $\Pi _{C}\mathcal{A}$ has the same number of $\omega $-chains and the same number of $\omega ^{ \ast }$-chains as $\mathcal{A}$.
\end{theorem}

\begin{proof}
Let $\mathcal{A} =(A ,f)$ and $\mathcal{B} =\Pi _{C}\mathcal{A}$. Let $M$ be the maximum size of finite chains. The property that $\mathcal{A}$ has $ \geq n$ many $\omega $-chains can be expressed by the following $\Sigma _{2}^{0}$ sentence: \bigskip

$ \exists x_{1} \cdots  \exists x_{n} \exists y_{1} \cdots  \exists y_{n} \forall z[\bigwedge _{1 \leq i <j \leq n}x_{i} \neq x_{j} \wedge \bigwedge _{1 \leq i \leq n}(f^{M}(x_{i}) =y_{i} \wedge f(z) \neq x_{i})]$ \bigskip 

\noindent Hence $\mathcal{A}$ and $\mathcal{B}$ have the same number of $\omega $-chains.

Similarly, the property that $\mathcal{A}$ has $ \geq n$ many $\omega ^{ \ast }$-chains can be expressed by the following $\Sigma _{2}^{0}$ sentence: \bigskip

$ \exists x_{1} \cdots  \exists x_{n} \exists y_{1} \cdots  \exists y_{n} \forall z[\bigwedge _{1 \leq i <j \leq n}x_{i} \neq x_{j} \wedge \bigwedge _{1 \leq i \leq n}(f^{M}(y_{i}) =x_{i} \wedge f(x_{i}) \neq z_{})]$.
\end{proof}

\bigskip

\begin{theorem}
 Let $\mathcal{A} =(A ,f)$ be a computable partial  injection structure with unbounded finite chain character. Let $C$ be a cohesive set. \medskip

(i) Assume that $ran(f)$ is computable. Then $\Pi _{C}\mathcal{A}$ has infinitely many $\omega $-chains. \medskip

Hence if $\mathcal{A}$ has only finitely many $\omega $-chains, then  $\mathcal{A} \ncong \Pi _{C}\mathcal{A}$. \bigskip 

(ii) Assume that $dom(f)$ is computable. Then $\Pi _{C}\mathcal{A}$ has infinitely many $\omega ^{ \ast }$-chains. \medskip

Hence if $\mathcal{A}$ has only finitely many $\omega ^{ \ast }$-chains, then  $\mathcal{A} \ncong \Pi _{C}\mathcal{A}$.
\end{theorem}

\begin{proof}
  (i)  Proof is similar to that of Theorem \ref{omegachains}.\medskip

(ii) Let $g :\omega  \rightarrow (A -dom(f))$ be a computable function such that if natural numbers $m_{1} ,m_{2} ,n_{1} ,n_{2}$ are such that $m_{1} \neq m_{2}$ or $n_{1} \neq n_{2} ,$ then $f^{ -n_{1}}(g(m_{1})) \neq f^{ -n_{2}}(g(m_{2}))$ where $n_{1} \leq m_{1}$ and $\left .n_{2}\right . \leq m_{2}$.
Such a function exists since the chain character is unbounded and $dom(f)$ is computable.

We define partial function $\psi _{m ,n} :\omega  \rightarrow A$ for $m ,n \in \omega $ as follows:\bigskip

$\psi _{m ,n}(x) =\left \{\begin{array}{cc}f^{ -n}(g( \langle m ,x \rangle )) & \text{if }n \leq  \langle m ,x \rangle  ; \\
\uparrow  & \text{otherwise.}\end{array}\right .$

\bigskip

\noindent It follows that $[\psi _{m ,n}] \in \Pi _{C}A$ since $dom(\psi _{m ,n})$ is cofinite. We can show that $\{[\psi _{m ,n}] :n \in \omega \}$ form an $\omega ^{ \ast }$-chains, and that for $m_{1} \neq m_{2}$ we have that $[\psi _{m_{1} ,0}]$ and $[\psi _{m_{2} ,0}]$ belong to different $\omega ^{ \ast }$-chains. Hence $\Pi _{C}\mathcal{A}$ has infinitely many $\omega ^{ \ast }$-chains.  
\end{proof}

\section{Concluding remarks
}

In this paper, we focus on the isomorphism types of cohesive powers of certain computable structures with a single binary relation, such as graphs, equivalence structures, and partial injection structures. This adds to the previous study of the cohesive powers of the ordered set of natural numbers, $(\omega  , <)$, and other natural linear orders in \cite{LinOrd1,LinOrd}. Here, we also investigate cohesive powers of computable structures with a unary function that is one-to-one, two-to-one, and $(2 ,0) :1$, which can be identified with the directed graphs they induce. It will be worthwhile to investigate the isomorphism types of cohesive powers of other directed graphs induced by functions. Some structures in the classes we consider are isomorphic to all of their cohesive powers. It was previously known that this is also true for finite structures, ordered set of rationals, random graph, and the countable atomless Boolean algebra. Some structures in the classes we consider are not isomorphic to their cohesive powers, having properties that distinguish them and that can be described by computable (infinitary) sentences.

Our goal is to further develop the theory of cohesive powers and, more generally, cohesive products of effective structures by investigating their algebraic, computability-theoretic, and syntactic properties. We would like to include more complicated algebraic structures such as semigroups, groups, rings, and fields. Cohesive powers of computable fields will extend the earlier study of cohesive powers of the field of rationals, $(Q , + ,) ,$ in \cite{DHMM} and will have further applications in the study of the lattice of c.e.\ vector spaces and their automorphisms, thus generalizing results in \cite{DH1}.

\end{document}